\documentclass[a4paper,12pt,reqno,twoside]{amsart}
\numberwithin{equation}{section}

\usepackage[latin1]{inputenc}
\usepackage[T1]{fontenc}
\usepackage{color}
\definecolor{MyLinkColor}{rgb}{0,0,0.4}

\newcommand{\R}{\mathbb R}
\newcommand{\Real}{\mathbb R}
\newcommand{\Z}{\mathbb Z}
\newcommand{\N}{\mathbb N}
\newcommand{\norm}[1]{\left\Vert#1\right\Vert}

\newcommand{\epsi}{\varepsilon}
\newcommand{\brac}[1]{\langle#1\rangle}
\renewcommand{\phi}{\varphi}
\newcommand{\abs}[1]{\left\vert#1\right\vert}

\newtheorem{thm}{Theorem}[section]

\newtheorem{lemma}[thm]{Lemma}

\newtheorem{definition}[thm]{Definition}

\theoremstyle{definition}
\newtheorem{rem}[thm]{Remark}

\title{Symmetric waves are traveling waves}


\author[Ehrnstr\"om]{Mats Ehrnstr\"om}
\address{Institut f{\"u}r Angewandte Mathematik, Leibniz Universit{\"a}t Hannover, Welfengarten~1, 30167 Hannover, Germany. }
\email{ehrnstrom@ifam.uni-hannover.de}
\thanks{ME gratefully acknowledges the support of the Swedish Royal Physiographic Society. This paper was written as part of  the international research program on Nonlinear Partial Differential Equations at the Centre for Advanced Study at the Norwegian Academy of Science and Letters in Oslo during the academic year 2008--09.}

\author[Holden]{Helge Holden}
\address{Department of Mathematical Sciences, Norwegian University of Science and Technology,  NO-7491 Trondheim, Norway, and
 Centre of Mathematics for Applications, 
 University of Oslo,
 P.O.\ Box 1053, Blindern,
 NO--0316 Oslo, Norway}
\email{holden@math.ntnu.no}

\author[Raynaud]{Xavier Raynaud}
\address{Department of Mathematical Sciences, Norwegian University of Science and Technology, NO-7491 Trondheim, Norway, and
 Centre of Mathematics for Applications, 
 University of Oslo,
 P.O.\ Box 1053, Blindern,
 NO--0316 Oslo, Norway}
\email{raynaud@math.ntnu.no}

\usepackage{hyperref} 

\begin{document}

\begin{abstract}
 We show that horizontally symmetric water waves
 are traveling waves. The result
 is valid for the Euler equations, and is based
 on a general principle that applies to a
 large class of nonlinear partial differential
 equations, including some of the most famous 
 model equations for water waves. A detailed analysis is 
given for weak solutions of the Camassa--Holm equation. 
In addition, we establish the
 existence of nonsymmetric linear rotational
 waves for the Euler equations.
\end{abstract}

\maketitle

%%% SECTION: INTRODUCTION %%%
\section{Introduction}\label{sec:introduction}

This paper is devoted to the relation between
symmetric and traveling water waves. The physical
setting is that of two-dimensional surface gravity
waves propagating in a perfect fluid over a flat bed (cf.~\cite{MR1629555}). 
The fluid is this thus assumed to be inviscid with constant density, 
and the effects of surface tension are neglected.

It is a striking fact that the traveling water waves known to exist are all
symmetric (cf., e.g.,
\cite{MR2027299,MR1422004}).  Symmetry is moreover 
\emph{a priori} guaranteed
for large classes of traveling waves, including those with a 
monotone surface profile between troughs and crests \cite{MR2362244}. 
This raises the intriguing question whether the classes of traveling 
and symmetric waves are identical.

This investigation gives a partial and affirmative answer to that question: we
show
that any horizontally symmetric wave by necessity
has to be a traveling, i.e. steady, wave. That means
that the axis of symmetry moves with constant
speed, and that the shape of the surface remains
unchanged. The result, which is based on a 
general principle, is valid under the assumption that the initial value
problem admits a unique solution. It
applies to a wide class of nonlinear differential
equations, among them the Euler equations, the Korteweg--de Vries, the Camassa--Holm, the Degasperis--Procesi, and the Benjamin--Bona--Mahoney equations.

On the other hand, we also show that within the linear approximation of the full water-wave problem, there do exist nonsymmetric traveling waves. The existence of steady nonsymmetric gravity waves for the full nonlinear problem has been a long-standing open problem. Although the waves we construct do not solve the full Euler equations, they point to the possible existence of their nonlinear counterparts. Thus, while the first result greatly strengthens the relation between symmetric and traveling water waves, the second one indicates that, after all, the two classes may not be identical. 

The paper has the following disposition. In Section \ref{sec:general} we prove the general principle, asserting that horizontally symmetric solutions to a large class of nonlinear evolution equations have to be traveling-wave solutions. For transparency, the proof is carried out for classical solution. The same idea applies to many equations where it is of interest to consider also weak solutions. For example, to include the important class of multipeakon solutions for the Camassa--Holm equation \cite{PhysRevLett.71.1661}, one has to extend the solution space beyond differentiable functions. In Section~\ref{sec:models} we therefore apply it to the Camassa--Holm equation, showing how to modify the argument to accommodate for weak solutions. Section~\ref{sec:euler} contains the corresponding proof for the full Euler equations, assuming that we are in a setting with unique classical solutions. 

Finally, in Section~\ref{sec:asymmetric} we discuss the possibility of nonsymmetric waves for the Euler equations by bifurcation from a nonsymmetric kernel. Some background on symmetry for exact water waves is given, and by means of inverse Sturm--Liouville theory, nonsymmetric solutions for the linearized Euler equations are constructed.

\section{A general principle}\label{sec:general}

The general principle that underlies the results in this paper is best studied in the general setting, assuming smooth solutions of the nonlinear partial differential equation. 

\begin{definition}
 A solution $u$  is $x$-symmetric if there exists a function $\lambda \in C^1(\Real_+)$ such that  for every $t>0$,
 \begin{equation*}
   u(t,x) = u(t, 2\lambda(t)-x)   
 \end{equation*}
 for almost every $x\in\Real$.
 We say that  $\lambda(t)$ is the axis of symmetry.
\end{definition}

Next we show that for a general class of differential equations, $x$-symmetric waves are indeed traveling waves.
%%% COROLLARY: GENERAL %%%
\begin{thm}\label{cor:general}
 Let $P$ be a polynomial. We consider the
 equation
 \begin{equation}\label{eq:general}
   P(\partial_x) u_t = F(u) = \bar F(u, \partial_x u, \ldots, \partial^n_x u), \qquad n \in \N, 
 \end{equation}
 and assume that this equation admits at most one classical
 solution $u(t,x)$ for given initial data
 $u(0,x)$. If $P$ is even and $\bar F$ satisfies
 \begin{equation}
   \label{eq:condF1}
   \bar F(a_0,-a_1,a_2,-a_3,\cdots)=-\bar F(a_0,a_1,a_2,a_3,\cdots)
 \end{equation}
 for all $a_i\in\Real$, or if $P$ is odd and $\bar F$ satisfies
 \begin{equation}
   \label{eq:condF2}
   \bar F(a_0,-a_1,a_2,-a_3,\cdots)=\bar F(a_0,a_1,a_2,a_3,\cdots)
 \end{equation}
 for all $a_i\in\Real$, then any $x$-symmetric solution
 of \eqref{eq:general} is a traveling wave
 solution.
\end{thm}

\begin{rem}
 Property \eqref{eq:condF1} can be rewritten as
 the following condition on $F$:
 \begin{equation}
   \label{eq:condF1b}
   F(u(-x))=-F(u)(-x)
 \end{equation}
 for all smooth functions $u\colon\Real\to\Real$ and
 $x\in\Real$, while property \eqref{eq:condF2} is
 equivalent to
 \begin{equation}
   \label{eq:condF2b}
   F(u(-x))=F(u)(-x)
 \end{equation}
 for all smooth functions $u \colon\Real\to\Real$ and
 $x\in\Real$. One can also check that
 $P(\partial_x)$ satisfies \eqref{eq:condF2b} if
 $P$ is even and \eqref{eq:condF1b} if $P$ is
 odd.
\end{rem} 

%%% REMARK: CLASSICAL SOLUTIONS %%%
\begin{rem}
 We formulate and prove
 Theorem~\ref{cor:general} for classical
 solutions, but whenever a weak formulation is
 available, a more technical argument is required. We illustrate this in the case of the Camassa--Holm equation in the next section, cf. Theorem~\ref{thm:CH}.  
\end{rem}

%%% REMARK: MODEL EQUATIONS %%%
\begin{rem}
 Some equations that fulfill the formal
 requirements of Theorem~\ref{cor:general} are
 the Korteweg--de Vries equation
 \[
 u_t + u_{xxx} + 6uu_x = 0,
 \]
 the Benjamin--Bona--Mahony equation,
 \[
 u_t - u_{xxt} + u_x + u u_x = 0,
 \]
 the Degasperis--Procesi equation,
 \[
 u_t - u_{xxt} + 2\kappa u_x + 4u u_x = 3 u_x u_{xx} + u u_{xxx},
 \]
 and the Kadomtsev--Petviashvili equation,
 \[
 (2u_t + 3 u u_x + \frac{1}{3} u_{xxx})_x + u_{yy} = 0.
 \]
\end{rem}

%%% PROOF OF COROLLARY: GENERAL %%%
\begin{proof}
 We pursue the proof for the case when $P$ is
 even and $F$ satisfies \eqref{eq:condF1b}; the
 proof of the other case is similar.  Observe
 first that if $\bar u(t,x)=U(x-ct)$ is a
 traveling wave, then
 \begin{align*}
   P(\partial_x)\partial_t\bar u=P(\partial_x)(-c\partial_xU(x-ct))=-cP(\partial_x)(\partial_xU)(x-ct)
 \end{align*}
 and
 \begin{align*}
   F(\bar u)=F(U(x-ct))=F(U)(x-ct).
 \end{align*}
 Thus if  $\bar u(t,x)$ is a solution, we find that $U$ satisfies
 \begin{equation}\label{eq:U}
   \big(P(\partial_x) \bar u_t - F(\bar u)\big)(t,x)= \big(- c P(\partial_x) U_x-F(U)\big)(x-ct)=0.
 \end{equation} 
 Thus $\bar u$ is a traveling wave solution if and only if $-cP(\partial_x)(\partial_xU)=F(U)$.

 Let now $u(t,x) = {u(t,2\lambda(t)-x)}$ be an $x$-symmetric solution of $P(\partial_x) u_t = F(u)$. Then 
 \begin{align*}
   0 &= (P(\partial_x) \partial_t - F)(u(t,x))\\ 
   &= (P(\partial_x) \partial_t - F)(u(t,2\lambda(t)-x))\\ 
   &= \left(P(\partial_x) (u_t + 2\dot\lambda(t) u_x) + F(u)\right) \big|_{(t,2\lambda(t)-x)}. 
 \end{align*} 
 where we have used the properties
 \eqref{eq:condF1b}  and \eqref{eq:condF2b} for
 $F$ and $P(\partial_x)$, respectively. In view of the fact that $x$ is
 arbitrary, we infer that
 \[
 P(\partial_x) u_t = F(u) = -P(\partial_x) (u_t + 2\dot\lambda u_x),
 \]
 and hence
 \[
 F(u) = -\dot\lambda P(\partial_x) u_x.
 \]
 Fix a time $t_0$, and define $c=\dot\lambda(t_0)$, and the function
 \begin{equation}
   \bar u(t,x)= u(t_0, x-c(t-t_0)).
 \end{equation}
 Then $\bar u$ is a traveling wave solution since it satisfies equation \eqref{eq:U}, and  it coincides with $u$ at $(t_0,x)$, that is, $\bar u(t_0,x)=u(t_0,x)$.
 From uniqueness with respect to initial
 data, it follows that $u(t,x) =u(t,2\lambda(t) - x) =
 u(t_0, x-c(t-t_0))=\bar u(t,x)$ for all $t$.  Observe that we find the explicit expression  $u(t,x)=U(x-ct)$ with
 $U(z)=u(t_0,z+ct_0)$. 
\end{proof}

%%% SECTION: MODELS %%%
\section{The Camassa--Holm equation}\label{sec:models}
The Camassa--Holm equation is given by
\begin{equation}\label{eq:CH}
 (1-\partial_x^2)u_t + 2 k u_x + u (1-\partial_x^2)u_x + 2 u_x (1-\partial_x^2) u = 0.
\end{equation}
This equation has been extensively studied due to
its amazing properties, suffice it to mention here
that it is completely integrable with a Lax pair
and infinitely many conserved quantities \cite{MR2125239}, and
enjoys wave breaking in finite time \cite{MR1668586}. The
continuation of solutions past wave breaking has
turned out to be subtle, and important
non-uniqueness issues have to be addressed \cite{MR2278406,MR2288533,1136.35080,HR09}.
Furthermore, its precise meaning as a model for
water waves has recently been clarified \cite{CL2009}, and it has an interesting geometric
interpretation \cite{MR1930889}. 

An important class of
solutions of the Camassa--Holm equation consists of
multipeakons. These solutions can be written as a
finite number of peakons, each behaving like $c
e^{-\abs{x-ct}}$ close to the peak, and
interacting in a particle-like fashion. Observe
that multipeakons are not differentiable, and
hence it is required to study weak solutions.

Let $\R_+ := (0,\infty)$. We use the following
definition of weak solutions of the Camassa--Holm equation.
\begin{definition}
 \label{def:defweak}
 A function $u(t,x)$ is a weak solution of the
 Camassa--Holm equation if $u\in
 C(\R_+,H^1(\Real))$ and
 \begin{equation}\label{eq:weakCH}
   \iint_{\Real_+\times\Real} \left[ u ((1-\partial_x^2) \varphi_t + 2k \varphi_x) + u^2 \left(\frac{3 \varphi_x }{2}  - \frac{\varphi_{xxx}}{2} \right) +  u_x^2 \frac{\varphi_x}{2}  \right] \, dt\, dx = 0, 
 \end{equation}
 for all $\phi\in C_0^\infty(\R_+\times\R)$.
\end{definition}
\begin{rem}
 \label{rem:sweqw}
 Since $C_0^\infty(\Real_+\times\Real)$ is dense
 in $C_0^1(\Real_+,C_0^3(\Real))$, one can prove
 by a density argument that Definition
 \ref{def:defweak} is unchanged if one only considers
 test functions in
 $C_0^1(\Real_+,C^3_0(\Real))$.
\end{rem}
The Camassa-Holm equation is not well-posed in
$C([0,T],H^1(\Real))$: Given $T>0$, one can
allways construct a solution $u(t,x)$ and a
sequence of solutions $u^n(t,x)$ such that
$\lim_{n\to\infty}u^n(0,\cdot)=u(0,\cdot)\text{ in
}H^1(\Real)$ and
$\norm{u^n-u}_{C([0,T],H^1(\Real))}\geq1$.
However, we have the following well-posedness
results when considering particular cases of
initial data:\\[2mm]
Case (i) If $u_0\in H^s$ for $s>3/2$ then there
exists $T>0$ depending only on $\norm{u_0}_{H^s}$
such that there exists a unique solution $u\in
C([0,T],H^s)$, cf. \cite{MR1851854}.\\[2mm]
Case (ii) If $u_0\in H^1$ and $u_0-u_{0,xx}$ is a
positive Radon measure, then there exists a unique
global solution, see \cite{1002.35101, MR2219220}.\\[2mm]
In both cases one has to consider weak
solutions. The first case excludes multipeakons as
they satisfy $u(t)\in H^s$ for $s<3/2$, while the
second case includes multipeakons of definite sign
(that is, only peakons, no antipeakons). Using the
bracket notation for distributions,
\eqref{eq:weakCH} rewrites as
\begin{equation}\label{eq:weakCH2}
 \brac{u,(1-\partial_x^2) \varphi_t + 2k \varphi_x} + \brac{u^2, \frac{3 \varphi_x }{2}  - \frac{\varphi_{xxx}}{2}} + \brac{u_x^2, \frac{\varphi_x}{2}}= 0. 
\end{equation}
We are interested in the weak solutions of the Camassa--Holm
equation which are $x$-symmetric. 
For such solutions, the following theorem holds.
%%% THEOREM: CH %%%
\begin{thm}\label{thm:CH}
 Let $u$ be a weak solution of the Camassa--Holm
 equation with initial data such that the
 equation is locally
 well-posed (case (i) or (ii) above). If $u$ is $x$-symmetric, then it is a traveling wave.
\end{thm}

The next lemma provides a sufficient condition on
the initial conditions of the traveling waves.
\begin{lemma}
 \label{lem:trav}
 If $U\in H^1(\R)$ satisfies
 \begin{equation}\label{eq:steadytwo}
   \int_\Real\Big( U(2k-c(1-\partial_x^2)) \psi_x +  U^2\big(\frac32\psi_x  -\frac12\psi_{xxx}\big) + \frac12U_x^2\psi_x\Big)\, dx= 0,
 \end{equation}
 for all $\psi\in C_0^\infty(\R)$, then $u$ given
 by
 \begin{equation}
   \label{eq:uUxmct}
   u(t,x)=U(x-c(t-t_0))
 \end{equation}
 is a weak solution of the Camassa--Holm equation, for any $t_0\in\Real$.
\end{lemma}
\begin{proof} We can assume without loss of
 generality that $t_0=0$. 
By using the Fourier transform, it is immediate that the translation map $a \mapsto U(x+a)$ is continuous  $\R \to H^1(\R)$. Since $t \mapsto c(t-t_0)$ is real analytic, it thus follows that $u$ as given by \eqref{eq:uUxmct} belongs to $C(\R,H^1(\R))$. 

 For any function $\phi\in
 C_0^\infty(\R_+\times\R)$, we have that
 \begin{equation*}
   \brac{u,\phi}=\brac{U,\phi_c},\  \brac{u²,\phi}=\brac{U²,\phi_c}
   \text{ and } \brac{u_x^2,\phi}=\brac{U_x^2,\phi_c}
 \end{equation*}
 where we denote
 \begin{equation*}
   \phi_c=\phi(t,x+ct).
 \end{equation*}
 The following commutation identities 
 \begin{equation}
   \label{eq:com1}
   (\phi_c)_t=(\phi_t)_c+c(\phi_x)_c
 \end{equation}
 and
 \begin{equation}
   \label{eq:com2}
   (\phi_c)_x=(\phi_x)_c
 \end{equation}
 are easily seen to be valid.
 Let us check that \eqref{eq:weakCH} holds. By
 using \eqref{eq:com1} and \eqref{eq:com2}, we
 obtain
 \begin{align}
   \notag
   \brac{u,(1-\partial_{xx})\partial_t\phi+2k\phi_x}&=\brac{U,\big((1-\partial_{xx})\partial_t\phi+2k\phi_x\big)_c}\\
   \label{eq:perpr1}
   &=\brac{U,(1-\partial_{xx})(\partial_t\phi_c-c\partial_x\phi_c)+2k\partial_x\phi_c}
 \end{align}
 and
 \begin{align}
   \notag
   \brac{u^2, \frac{3 \varphi_x }{2}  - \frac{\varphi_{xxx}}{2}}+\brac{u_x^2, \frac{\varphi_x}{2}}&=\brac{U^2, \left(\frac{3 \varphi_x }{2}  - \frac{\varphi_{xxx}}{2}\right)_c}+\brac{U_x^2, (\frac{\varphi_x}{2})_c}\\
   \label{eq:perpr2}
   &=\brac{U^2, \frac{3 \partial_x\varphi_c }{2}  - \frac{\partial_{x}^3\varphi_c}{2}}+\brac{U_x^2, \frac{\partial_x\varphi_c}{2}}.
 \end{align}
 Since $U$ is independent of time, we obtain that
 \begin{align*}
   \brac{U,(1-\partial_{xx})\partial_t\phi_c}&=\int_{\Real}U(x)\int_{\Real_+}\partial_t(1-\partial_{xx})\phi_c\,dt\,dx\\
   &=\int_{\Real}U(x)[(1-\partial_{xx})\phi_c(T,x)-(1-\partial_{xx})\phi_c(0,x)]\,dx\\
   &=0
 \end{align*}
 for $T$ large enough so that it does not belong
 to the support of $\phi_c$.  Collecting the above results, we find

 {\allowdisplaybreaks
   \begin{align*} 
     & \brac{u,(1-\partial_x^2) \varphi_t + 2k \varphi_x} + \brac{u^2, \frac{3 \varphi_x }{2}  - \frac{\varphi_{xxx}}{2}} +  \brac{u_x^2, \frac{\varphi_x}{2}} \\
     &\quad =   \brac{U,(1-\partial_{xx})(\partial_t\phi_c-c\partial_x\phi_c)+2k\partial_x\phi_c} \\
     &\qquad+  \brac{U^2, \frac{3 \partial_x\varphi_c }{2}  - \frac{\partial_{x}^3\varphi_c}{2}}+\brac{U_x^2, \frac{\partial_x\varphi_c}{2}} \\
     &\quad =   \brac{U,(2k-c(1-\partial_{xx}))\partial_x\phi_c}+    \brac{U^2, \frac{3 \partial_x\varphi_c }{2}  - \frac{\partial_{x}^3\varphi_c}{2}}+\brac{U_x^2, \frac{\partial_x\varphi_c}{2}}\\
     & \quad =\int_{\Real_+}\int_\Real\bigg(U(x)(2k-c(1-\partial_{xx}))\partial_x\phi_c(t,x)
     \\
     &\qquad+U^2(x)(\frac{3 \partial_x\varphi_c(t,x)}{2}-\frac{\partial_{x}^3\varphi_c(t,x)}{2})+U_x^2(t,x)\frac{\partial_x\varphi_c(t,x)}{2}\bigg)\,dx\,dt\\
     & \quad =0
   \end{align*}  
 }
 by using \eqref{eq:steadytwo} with
 $\psi(x)=\phi_c(t,x)$, which, for each given
 $t\geq0$, belongs to $C_0^\infty(\Real)$. Hence,
 \eqref{eq:weakCH} is satisfied and the lemma
 is proven.
\end{proof}
%%% REMARK: LENELLS %%%
\begin{rem}
 In a series of papers
 \cite{1069.35072,1082.35127,pre05052844},
 Lenells has characterized the traveling wave
 solutions of some model equations, among them
 the Camassa--Holm equation. The solutions found include
 very exotic shapes, but when restricted to
 smooth solutions, the situation is clearer. In
 particular, the smooth solutions are all
 symmetric around the crest, wherefrom they decay
 either to the trough (in the periodic case), or
 to a flat profile at infinity (in the case of
 solitary waves). Though not stated as a result
 this is mentioned in passing in
 \cite[p. 410]{1082.35127}. It follows from the
 fact that, after integration, the steady solutions
 of the Camassa--Holm equation, $\varphi(x)$, satisfy
 \[
 \varphi_x = \frac{\varphi^2 ( c - 2k - \varphi) +
   a\varphi + b}{c-\varphi},
 \]
 where $a,b \in \R$ are integration constants. As a
 consequence of this result, for smooth enough waves we obtain
 identity between traveling and symmetric waves
 for the Camassa--Holm equation.
\end{rem}

%%% PROOF OF THEOREM:CH %%%
\begin{proof}[Proof of Theorem \ref{thm:CH}]
 We introduce the notation
 \begin{equation*}
   \varphi_\lambda(t,x)= \varphi(t, 2\lambda(t)-x)  
 \end{equation*}
 Consider test functions $\phi\in
 C_0^1(\Real_+,C_0^3(\Real))$. As noted in Remark
 \ref{rem:sweqw}, equation \eqref{eq:weakCH2}
 remains valid for such test functions. The space
 $C_0^1(\Real_+,C_0^3(\Real))$ is invariant under
 the transformation $\phi\mapsto\phi_\lambda$
 because $\lambda\in C^1(\Real)$. This
 transformation is a bijection as we have
 $((\phi)_\lambda)_\lambda=\phi$.  If $u$ is an
 $x$-symmetric solution of the Camassa--Holm equation, we
 have that
 \begin{equation*}
   \brac{u,\phi}=\brac{u_\lambda,\phi}=\brac{u,\phi_\lambda}.
 \end{equation*}
 Similar identities hold for $u^2$ and
 $u_x^2$. Then, \eqref{eq:weakCH2} implies that
 \begin{multline}
   \label{eq:chwl}
   \brac{u,\left((1-\partial_x^2)\partial_t \varphi + 2k\partial_x\varphi\right)_\lambda}\\ + \brac{u^2, \left(\frac{3 \partial_x\varphi}{2}  - \frac{\partial_x^3\varphi}{2}\right)_\lambda} +  \brac{u_x^2, \left(\frac{\partial_x\varphi}{2}\right)_\lambda}= 0
 \end{multline}
 The following commutation rules,
 \begin{equation}
   \label{eq:lcom1}
   \partial_t(\phi_\lambda)=(\partial_t\phi)_\lambda-2\dot\lambda(\partial_x\phi)_\lambda
 \end{equation}
 and
 \begin{equation}
   \label{eq:lcom2}
   \partial_x(\phi_\lambda)=-(\partial_x\phi)_\lambda,
 \end{equation}
 hold, where $\dot\lambda$ denotes the time derivative of $\lambda$.  By using \eqref{eq:lcom1} and \eqref{eq:lcom2},
 \eqref{eq:chwl} yields
 \begin{multline}
   \label{eq:chwlam}
   \brac{u,(1-\partial_x^2)\partial_t\phi_\lambda-2\dot\lambda(1-\partial_x^2)\partial_x\phi_\lambda-2k\partial_x\phi_\lambda}\\+\brac{u^2,-\frac{3\partial_x\phi_\lambda}2+\frac{\partial_x^3\phi_\lambda}2}+\brac{u_x^2,-\frac{\partial_x\phi_\lambda}2}=0.
 \end{multline}
 Hence, by taking $\phi$ equal to $\phi_\lambda$
 in \eqref{eq:chwlam}, as
 $(\phi_\lambda)_\lambda=\phi$, we obtain
 \begin{multline}
   \label{eq:chlnol}
   \brac{u,(1-\partial_x^2)\partial_t\phi-2\dot\lambda(1-\partial_x^2)\partial_x\phi-2k\partial_x\phi}\\+\brac{u^2,-\frac32\partial_x\phi+\frac12\partial_x^3\phi}+\brac{u_x^2,-\frac12\partial_x\phi}=0.
 \end{multline}
 After substracting \eqref{eq:chlnol} from
 \eqref{eq:weakCH2}, we get
 \begin{multline}
   \label{eq:sublor}
   \brac{u,2\dot\lambda(1-\partial_{xx})\partial_x\phi+4k\partial_x\phi}\\
   +\brac{u^2,3\partial_x\phi-\partial_x^3\phi}+\brac{u_x^2,\partial_x\phi}=0.
 \end{multline}
 We consider a fixed but arbitrary time $t_0>0$. For
 any $\psi\in C_0^\infty(\Real)$, let us consider
 the sequence of functions
 $\phi_\epsi(t,x)=\psi(x)\rho_\epsi(t)$ where
 $\rho_\epsi\in C_0^\infty(\Real_+)$ is a
 mollifier with the property that $\rho_\epsi$
 tends to $\delta(t-t_0)$, the Dirac mass at
 $t_0$, when $\epsi$ tends to zero. From
 \eqref{eq:sublor}, by using the test function
 $\phi_\epsi$, we get
 \begin{equation}\label{eq:rhoepsi}
   \begin{aligned}
     &\int_\Real\left(      2(1-\partial_{xx})\partial_x\psi \int_{\Real_+} \dot\lambda\, u(t,x)\rho_\epsi(t)\,dt \right)\,dx\\ 
     &\quad+ \int_\Real\left(  4k\partial_x\psi \int_{\Real_+}u(t,x)\rho_\epsi(t)\,dt    \right)\,dx\\
     &\quad+\int_\Real\left((3\partial_x\psi-\partial_x^3\psi)\int_{\Real_+}u^2(t,x)\rho_\epsi(t)\,dt\right)\,dx\\
     &\quad+\int_\Real\left(\frac12\partial_x\psi\int_{\Real_+}u_x^2(t,x)\rho_\epsi(t)\,dt\right)\,dx=0.
   \end{aligned}
 \end{equation}
 Since, by assumption, $u\in
 C(\Real_+,H^1(\Real))$, we have that
 \begin{equation*}
   \lim_{\epsi\to0}\int_{\Real_+}u(t,x)\rho_\epsi(t)\,dt=u(t_0,x) 
 \end{equation*}
 in $L^2(\Real)$ and 
 \begin{equation*}
   \lim_{\epsi\to0}\int_{\Real_+}u^2(t,x)\rho_\epsi(t)\,dt=u^2(t_0,x)
 \end{equation*}
 and
 \begin{equation*}
   \lim_{\epsi\to0}\int_{\Real_+}u_x^2(t,x)\rho_\epsi(t)\,dt=u_x^2(t_0,x)
 \end{equation*}
 in $L^1(\Real)$. Therefore, by letting $\epsi$
 tend to zero, \eqref{eq:rhoepsi} implies that
 $u(t_0,x)$ satisfies \eqref{eq:steadytwo} for
 $c=\dot\lambda(t_0)$. Lemma \ref{lem:trav} then
 implies that $\bar
 u(t,x)=u(t_0,x-\dot\lambda(t_0)(t-t_0))$ is a
 traveling wave solution of the Camassa--Holm
 equation. Since $\bar u$ and $u$ coincide for
 $t=t_0$, we have, by uniqueness of the solution,
 that $\bar u(t,x)=u(t,x)$ for all time $t$, and
 $u$ is a traveling wave solution.
\end{proof}

%%% SECTION: EULER %%%
\section{The Euler equations}\label{sec:euler}
In this section we prove that the assertion of Theorem~\ref{cor:general} holds for the full water wave problem. For classical meaning of the analysis to come, let the functions $u,v,P \in C^1(\R_+,C^2(\R))$ and $\eta \in C^1(\R_+ \times \R)$. The governing equations for a perfect fluid include the Euler equations,
\begin{equation}\label{eq:euler}
 \begin{aligned}
   u_t + u u_x + v u_y &= - P_x,\\
   v_t + u v_x + v v_y &= - P_y - g,
 \end{aligned}
\end{equation}
together with the assumption of incompressibility,   
\begin{equation}\label{eq:incompressible}
 u_x + v_y = 0,
\end{equation}
in the interior of the fluid domain. In addition we require that
\begin{align}
 v &= \eta_t + \eta_x u \label{eq:kinematic},\\
 P &= P_0\label{eq:dynamic},
\end{align}
on the free surface $y=\eta(t,x)$, and that
\begin{equation}\label{eq:bed}
 v=0
\end{equation}
along the flat bed $y=-d$. In total those equations describe the gravity-governed motion of an incompressible fluid of constant density, propagating on a flat bed with no mixing of air and liquid (for more details, cf. \cite{MR1629555}).

%%% THEOREM: EULER %%%
\begin{thm}\label{thm:euler}
 Any horizontally symmetric solution of the exact water wave problem~\eqref{eq:euler}--\eqref{eq:bed} constitutes a steady wave.
\end{thm}

%%% REMARK: EULER %%%
\begin{rem}\label{rem:euler}
 The proof of Theorem~\ref{thm:euler} is carried out only for classical solutions as defined below, and for a particular setting of the Euler equatons. There is no doubt this can be generalized. As in the case of the Camassa--Holm equation, the proof relies upon well-posedness results, which---for the Euler equations---come in very many different forms. A pioneering paper in this area was  \cite{0892.76009}, written for deep-water waves.  For two-dimensional finite-depth gravity waves, suitable results are \cite{MR2138139} (irrotational waves) and \cite{MR2002401} (rotational waves).
\end{rem}

In the setting of the Euler equations, symmetry
has a somewhat extended meaning compared to the
model equations. 
\begin{definition}
 A solution $(u,v,P,\eta)$ is $x$-symmetric if
 there exists a function $\lambda \in
 C^1(\Real_+)$ such that for every $t>0$,
 \begin{equation} \label{eq:symm0}
   \begin{aligned}
     u(t,x,y) &= u(t,2\lambda(t)-x,y),\\ 
     v(t,x,y) &=- v(t,2\lambda(t)-x,y),\\ 
     P(t,x,y) &= P(t,2\lambda(t)-x,y),\\
     \eta(t,x) &= \eta(t,2\lambda(t) -x).
   \end{aligned}
 \end{equation} 
 for almost every $x\in\Real$.  We say that
 $\lambda(t)$ is the axis of symmetry.
\end{definition}

Typically, model equations approximate either the
free surface or the horizontal velocity along the
surface (or at some fixed depth). Therefore,
horizontal symmetry of the wave means evenness of
the corresponding solution. Let us prove
Theorem~\ref{thm:euler}.

%%% PROOF OF THEOREM: EULER %%%
\begin{proof}[Proof of Theorem~\ref{thm:euler}]
 From \eqref{eq:symm0}, we get
 \begin{equation} 
   \label{eq:symm1}
   \begin{aligned}  
     u_t(t,x,y)&=  u_t(t,2\lambda(t)-x,y)+2\dot\lambda(t) u_x(t,2\lambda(t)-x,y), \\
     u_x(t,x,y) &= -u_x(t,2\lambda(t)-x,y), \\
     u_y(t,x,y) &= u_y(t,2\lambda(t)-x,y),
   \end{aligned} 
 \end{equation}  
 which implies that
 \begin{equation}\label{eq:symm2}
   u_t(t,2\lambda(t)-x,y)= u_t(t,x,y)+2\dot\lambda(t) u_x(t,x,y).
 \end{equation}  
 If we start by considering the first Euler equation $u_t+uu_x+vu_y=-P_x$, and evaluate it at the point $(t,2\lambda(t)-x,y)$, and then use \eqref{eq:symm0}--\eqref{eq:symm2}  as well as  $P_x(t,x,y) = -P_x(t,2\lambda(t)-x,y)$ to return to the variables $(t,x,y)$, we find
 \begin{equation}\label{eq:symm3}
   u_t+2\dot\lambda(t) u_x-u u_x-v u_y= P_x,
 \end{equation}  
 all evaluated at the point $(t,x,y)$.
 By subtracting  $u_t+uu_x+vu_y=-P_x$ at $(t,x,y)$, we find
 \begin{equation}\label{eq:symm4}
   (u-\dot\lambda(t))u_x+v u_y=- P_x
 \end{equation}
 evaluated at the point $(t,x,y)$.  
 By doing the same operations on the second Euler equation, we find
 \begin{equation}\label{eq:symm5}
   (u-\dot\lambda(t))v_x+v v_y=- P_y-g
 \end{equation}  
 at the point $(t,x,y)$.
 For the boundary condition at $y=\eta$ we find
 \begin{equation}\label{eq:symm6}
   v(t,x,y)=(u(t,x,y)-\dot\lambda)\eta_x(t,x,y).
 \end{equation}  

 Fix a time $t_0$ and define $c=\dot\lambda(t_0)$.  Introduce functions
 \begin{equation} \label{eq:symm7}
   \begin{aligned}
     \bar u(x,y) &= u(t_0,x,y),\\ 
     \bar v(x,y) &=v(t_0,x,y),\\ 
     \bar P(x,y) &= P(t_0,x,y),\\
     \bar \eta(x) &= \eta(t_0,x).
   \end{aligned}
 \end{equation} 
 By definition these functions satisfy
 \begin{equation} \label{eq:symm8}
   \begin{aligned}
     (\bar u-c) \bar u_x+\bar v \bar u_y &= -\bar P_x,\\
     (\bar u-c) \bar v_x+\bar v \bar v_y &= -\bar P_y-g,\\ 
     \bar v &=(\bar u-c)\bar\eta, \quad \bar P=P_0 \text{  at $y=\bar\eta$},\\ 
     \bar v &= 0 \text{  at $y=-d$}.
   \end{aligned}
 \end{equation} 
 Finally define the functions
 \begin{equation} \label{eq:symm8a}
   \begin{aligned}
     \tilde u(t,x,y) &= \bar u(x-c(t-t_0),y),\\ 
     \tilde v(t,x,y) &= \bar v(x-c(t-t_0),y),\\ 
     \tilde P(t,x,y) &= \bar P(x-c(t-t_0),y),\\ 
     \tilde \eta(t,x) &= \bar \eta(x-c(t-t_0)).
   \end{aligned}
 \end{equation} 
 By construction we have
 \begin{multline}\label{eq:symm9}
   ( \tilde u(t_0,x,y), \tilde v(t_0,x,y),\tilde P(t_0,x,y),\tilde \eta(t_0,x))\\
   =(  u(t_0,x,y), v(t_0,x,y), P(t_0,x,y), \eta(t_0,x)),
 \end{multline}  
 and $\tilde u, \tilde v, \tilde P, \tilde \eta$ will satisfy the Euler equations. By uniqueness of the solution of the Euler equations we conclude that
 $(\tilde u,\tilde v, \tilde P, \tilde\eta)=(u, v, P, \eta)$ everywhere. 
\end{proof}

%%% SECTION: ASYMMETRIC %%%
\section{Existence of nonsymmetric linear waves}\label{sec:asymmetric}
Do two-dimensional steady gravity water waves and two-dimensional symmetric gravity water waves constitute one and the same class? By studying the linear steady solutions we shall see that, for constant vorticity, there is a strong case for this notion. However, we also find that there exist flows of nonconstant vorticity for which the linear problem admits nonsymmetric solutions. 

We note that symmetry is \emph{a priori} guaranteed for large classes of steady waves: irrotational periodic waves with a particular interior structure \cite{MR1734884}; irrotational solitary waves \cite{MR919444}; internal waves \cite{MR1416035}; and a major class of rotational waves  \cite{MR2256915}. The latter results was recently extended to include waves of arbitrary vorticity either carrying internal structure \cite{Hur2007}, or having a monotone surface profile between troughs and crests \cite{MR2362244}. Similar results are available also in the case of infinite depth. 

For waves of small amplitude, the linearized steady Euler equations provide a good understanding of the exact waves. Indeed the exact steady waves found in \cite{MR2027299} are, at least near the trivial flows, all perturbations of linear waves, and their properties have been found to match very well those of their linear approximations \cite{MR2412330}. 

In \cite{EV08} the deduction for steady waves linearized around a laminar flow is presented. Let $U(y)$ be a function describing the underlying current into which the wave propagates, that is, the profile $\{ U(y) \colon 0 \leq y \leq 1 \}$ is the velocity profile of a running stream with a flat surface. Then write a general solution of the steady Euler equations as a small perturbation of the running stream $U(y)$:
\begin{equation*}
 u = U + \varepsilon \tilde u, \quad v = \varepsilon \tilde v, \quad p = \varepsilon \tilde p.
\end{equation*}
By first inserting this into the steady Euler equations, and then letting $\varepsilon \to 0$, 
we obtain the linearized system  (where we have replaced $\tilde u$ by $u$, etc.)
\begin{subequations}\label{eqs:final}
 \begin{align}
   u_{x} + v_{y} &= 0,\\
   (U-c) u_{x} + v U_{y} &= - p_{x},\\
   (U-c) v_{x}   &= -p_{y},
 \end{align}
 valid for $0 < y < 1$, and
 \begin{equation}
   v = (U-c) \eta_{x}, \quad\text{ and }\quad p = \eta, 
 \end{equation}
 valid for $y = 1$. Moreover
 \begin{equation}
   v= 0
 \end{equation}
\end{subequations}
when $y = 0$. Here $\eta$ is the disturbance of the nondimensionalized flat surface $y = 1$, and $c$ is the group speed of the wave (the case when $U \equiv c$  admits only laminar flows). The disturbance of the surface should have a vanishing mean over each period,
\begin{subequations}\label{eq:normalization} 
 \begin{equation}
   \int_{-L/2}^{L/2} \eta(x)\,dx = 0,
 \end{equation}
 $L$ denoting wavelength, and $u$ should have no component depending only on $y$,
 \begin{equation}
   u(x,y) \equiv \int_0^x u_x(\xi,y)\,d\xi,
 \end{equation}
\end{subequations}
or equivalently $u(0,y)=0$, 
since that should be part of the background current $U(y)$. For such linear solutions, the following result holds:

%%% THEOREM: asymmetric %%%
\begin{thm}\label{thm:asymmetric}
 $ $
 \begin{itemize}
 \item[(i)] If the background current $U(y)$ is of constant vorticity,  ${U^{\prime\prime} = 0}$, and different from the wave speed, ${U(y) \neq  c}$ for all $y \in [0,1]$, then the linear problem admits only symmetric solutions. 
 \item[(ii)]
   There exists an $a.e.$ twice differentiable background current $U(y)$, such that the 
   the linearized problem has multiple solutions and, in particular, non-symmetric solutions.     
 \end{itemize}
\end{thm}

\begin{rem}
 This has the following meaning for the full water wave problem: when the vorticity is constant, there are no asymmetric waves close to the trivial laminar solution. There exists, on the other hand, particular background currents that may allow for bifurcation from asymmetric kernels. Thus, one cannot at this point exclude the existence of symmetry-breaking bifurcations.  
\end{rem}

%%% PROOF OF THEOREM: asymmetric %%%
\begin{proof}[Proof of Theorem~\ref{thm:asymmetric}]
 By taking the curl of the linearized Euler equations, and
 by differentiating $p = \eta$ along the
 linearized surface $y=1$, it is easy to see that
 \begin{equation}\label{eq:vsystem}
   \begin{aligned}
     (U-c) \left(v_{xx} + v_{yy}\right) &= U_{yy} v, \qquad &0<y < 1,\\
     (1+(U-c)U_{y}) v &= (U-c)^2 v_{y}, \qquad	&y =1,\\
     v &= 0, \qquad &y=0.
   \end{aligned}  
 \end{equation}
 The system \eqref{eq:vsystem} is equivalent to
 \eqref{eqs:final} in the sense that if
 $(u,v,p,\eta)$ is a solution of the first
 system, then $v$ fulfills \eqref{eq:vsystem},
 and if $v$ is a solution of \eqref{eq:vsystem},
 then one can find $(u,p,\eta)$ such that
 \eqref{eqs:final} holds. While for a given $v$,
 a solution $u$ is only determined modulo
 functions $f(y)$, and $\eta$ up to a constant,
 the prescribed normalization
 \eqref{eq:normalization} leads to
 uniqueness. 

 By rescaling in $x$, and letting $\alpha(y) = U^{\prime\prime}(y)/(U(y)-c)$, we may consider the system
 \begin{align*}
   \Delta v &= \alpha(y) v, \qquad &0<y < 1,\\
   \mu_1 v &= \mu_2 v_{y}, \qquad	&y =1,\\
   v &= 0, \qquad &y=0,
 \end{align*}  
 where $\mu_1, \mu_2 \in \R$ with $\mu_1^2 + \mu_2^2 \neq 0$. 

 \bigskip
 \emph{Case (i).} We have that $\alpha \equiv 0$, hence $\alpha \in C^\beta((0,1))$, for any $\beta \in (0,1)$. Then necessarily $v \in C^{2,\beta}(\R \times (0,1))$ \cite{MR1814364}. That guarantees that the subsequent analysis makes sense pointwise. A Fourier series expansion $v = \sum_{k\in\Z} f_{k}(y) \exp(ikx)$ leads to
 \begin{equation}
   \begin{aligned}\label{eq:sturm}
     -f_{k}^{\prime\prime} + \alpha(y) f_{k} &= \lambda_k f_{k},\\
     \mu_1 f_{k}(1) &= \mu_2 f_{k}^\prime(1),\\
     f_{k}(0) &= 0,
   \end{aligned}
 \end{equation}
 with $\lambda_k = -k^2$. Then \eqref{eq:sturm} is a regular Sturm--Liouville problem, and according to standard theory \cite{MR894477} there exists a number $\lambda_{0}$ such that there are no eigenvalues $\lambda_k < \lambda_{0}$. We find that $f_k(y) = \sinh(ky)$ are the only possible solutions of \eqref{eq:sturm}.  It is easy to see from the boundary conditions that there is at most one integer $k$ admitting a solution, and this happens only if $U$ satisfies 
 \begin{equation}\label{eq:dispersion}
   \frac{(U(1)-c)^2}{1+(U(1)-c)U^\prime(1)} = \tanh{k^2},
 \end{equation}
 for some $k \in \Z$. Since traveling waves are invariant under translations in $x$, there is no loss of generality in requiring that $v(0,1) = 0$, meaning that at $x=0$ there is a crest or a trough. For running streams of constant vorticity the kernel of the linear problem thus is a one-dimensional family of solutions
 \[
 v(x,y) = \sin(kx) \sinh(ky),
 \]
 where $k$ is given by the dispersion relation \eqref{eq:dispersion}. It then follows from \eqref{eqs:final} that $(u,p,\eta)$ is an even function in $x$ so that the wave is symmetric.

 \bigskip
 \emph{Case (ii).}
 Consider the system \eqref{eq:sturm}, where $\alpha(y)$ is now unknown. We shall make use of the following result from inverse spectral theory.
 \begin{lemma}\cite{MR733718}\label{lemma:inverse}
   Fix $\mu_2 \neq 0$ and $\mu_1 \in \R$. A sequence of real numbers $\{\lambda_k\}_{k\geq 0}$ is the spectrum of a Sturm--Liouville problem \eqref{eq:sturm}, for some $\alpha \in L^2([0,1])$, if and only if $\lambda_k$ is an increasing sequence, and
   \[
   \lambda_k = \left(k+\frac{1}{2}\right)^2 \pi^2 + C + r_k, 
   \]
   for some $C \in \R$, $\{r_k\}_{k\geq 0} \in l^2(\R)$.
 \end{lemma}

 We may thus first fix $\mu_1 \in \R$, $\mu_2 >
 0$, and a finite $N \geq 2$. By choosing $C =
 0$, and
 \[
 \begin{cases}
   r_k := -k^2 - \left(k+\frac{1}{2}\right)^2 \pi^2, \qquad &k = 0,1, \ldots, N,\\
   r_k := 0, \qquad &k \geq N+1,
 \end{cases}
 \]  
 in Lemma~\ref{lemma:inverse}, we obtain $\lambda_k = -k^2$ for $k = 0,1 \ldots, N$, and conclude that for any finite number of $k$'s there is a function $\alpha(y)$, and nontrivial functions $f_k(y)$, satisfying \eqref{eq:sturm} (since $N$ is a finite number, there is no problem rearranging the first $N$ eigenvalues in increasing order). By defining a real function $v$ as a linear combination of such, we obtain that 
 \begin{equation}\label{eq:vgeneral}
   v(x,y) = \sum_{k=0}^N f_{k}(y) (a_k \sin(kx) + b_k \cos(kx)),
 \end{equation}
 where $N$ is always finite. The $a.e.$ twice
 differentiable function $v$ is then a solution
 of \eqref{eq:sturm} in the weak sense (it
 satisfies the equations pointwise $a.e.$).

 Furthermore, by solving the second order differential equation 
 \[
 U^{\prime\prime}(y) = \alpha(y) (U(y)-c), 
 \]
 with initial value conditions 
 \[
 U(1) = c + \sqrt{\mu_2} \qquad\text{ and }\qquad U^\prime(1) = \frac{\mu_1 -1}{\sqrt{\mu_2}},
 \]
 we recover the background current $U(y)$.  
 Let $x=0$  be the position of a
 crest or trough. It means that we
 impose $v(0,1)=0$, which is equivalent to
 \begin{equation}
   \label{eq:conbk}
   \sum_{k=0}^N f_k(1) b_k = 0.  
 \end{equation}
 If the solution is symmetric, then we have
 $v(x,y)=-v(-x,y)$, that is, for the solution of
 the form \eqref{eq:vgeneral} that we are
 considering,
 \begin{equation}
   \label{eq:condsym}
   \sum_{k=0}^N f_{k}(y) b_k \cos(kx)=0
 \end{equation}
 for all $x\in\Real$ and almost all
 $y\in[0,1]$. Since the function $f_k$ form a
 basis, \eqref{eq:condsym} implies that
 $b_k\cos(kx)=0$ for all $x$, which in turn
 implies that $b_k=0$. For $N\geq2$, it is clear
 that there exist $b_k$'s which satisfy
 \eqref{eq:conbk} and which are not all equal to
 zero and, for those $b_k$, the solution is not
 symmetric. Hence there exist asymmetric
 solutions of the linearized Euler equations.
\end{proof}

%%% BIBLIOGRAPHY %%%

\end{document}